\documentclass[12pt]{amsart}

\usepackage{amssymb}
\usepackage{bm}
\usepackage{graphicx}
\usepackage{enumerate}
\usepackage{psfrag}
\usepackage{color}
\usepackage{url}

\usepackage{algpseudocode}

\newcommand{\excise}[1]{}

\newtheorem{thm}{Theorem}[section]
\newtheorem{lemma}[thm]{Lemma}

\newtheorem{cor}[thm]{Corollary}
\newtheorem{prop}[thm]{Proposition}

\newtheorem{prob}[thm]{Problem}

\theoremstyle{definition}
\newtheorem{example}[thm]{Example}
\newtheorem{remark}[thm]{Remark}
\newtheorem{defn}[thm]{Definition}

\newtheorem{notation}[thm]{Notation}

\numberwithin{equation}{section}

\newcommand{\ring}[1]{\ensuremath{\mathbb{#1}}}

\renewcommand\>{\rangle}

\newcommand\NN{\ring{N}}

\newcommand\ZZ{\ring{Z}}

\newcommand\kk{\Bbbk}

\newcommand\oQ{\hspace{.15ex}\ol{\hspace{-.15ex}Q\hspace{-.25ex}}\hspace{.25ex}}
\newcommand\ttt{\mathbf{t}}

\newcommand\app{\mathord\approx}

\newcommand\til{\mathord\sim}

\newcommand\minus{\smallsetminus}

\def\ol#1{{\overline {#1}}}

\DeclareMathOperator\MesoAss{MesoAss} 

\begin{document}

\mbox{}
\title[On mesoprimary decomposition of monoid congruences]{On mesoprimary decomposition \\ of monoid congruences\qquad}

\author{Christopher O'Neill}
\address{Mathematics\\University of California, Davis\\One Shields Ave\\Davis, CA 95616}
\email{coneill@math.ucdavis.edu}

\makeatletter
  \@namedef{subjclassname@2010}{\textup{2010} Mathematics Subject Classification}
\makeatother

\begin{abstract}
\hspace{-2.05032pt}
We prove two main results concerning mesoprimary decomposition of monoid congruences, as introduced by Kahle and Miller.  First, we identify which associated prime congruences appear in every mesoprimary decomposition, thereby completing the theory of mesoprimary decomposition of monoid congruences as a more faithful analog of primary decomposition.  Second, we answer a question posed by Kahle and Miller by characterizing which finite posets arise as the set of associated prime congruences of monoid congruences.  
\end{abstract}
\maketitle


\section{Introduction}\label{s:intro}

A \emph{congruence} is an equivalence relation on the elements of a monoid that respects the monoid operation.  This paper focuses on congruences on the monoid of monomials in a polynomial ring $S$ that arise from \emph{binomial ideals} in $S$ (that is, ideals whose generators have at most two terms).  In particular, any binomial ideal $I \subset S$ identifies, up to scalar multiple, any two monomials appearing in the same binomial in $I$, inducing a congruence $\til_I$ on the monoid of monomials in $S$.  In~\cite{kmmeso}, Kahle and Miller introduce \emph{mesoprimary decompositions} of binomial ideals, which are combinatorial approximations of primary decompositions constructed from the underlying congruence.  

Mesoprimary decompositions are constructed in two settings: first for monoid congruences, and then for binomial ideals; both are designed to parallel standard primary decomposition in a Noetherian ring \cite[Chapter~3]{Eis96}.  At the heart of mesoprimary decomposition, for both monoid congruences and binomial ideals, lies a notion of associated objects analogous to associated prime ideals in standard primary decomposition.  In particular, any congruence $\til$ has a collection of associated prime congruences, and each component in a mesoprimary decomposition for $\til$ has precisely one associated prime congruence.  
However, unlike standard primary decomposition, eliminating redundant mesoprimary components can produce decompositions in which some of the associated objects do not appear as the associated object of any component (Example~\ref{e:keynottrue}).  

The focus of this paper is on mesoprimary decomposition of monoid congruences, and the two main results are as follows.  
First, we identify the class of truly associated prime congruences (Definition~\ref{d:truewitness}), which must appear as the associated prime congruence of some component in every mesoprimary decomposition of $\til$ (Theorems~\ref{t:truedecomp} and~\ref{t:trulyassociatedcong}), thereby completing the theory of mesoprimary decomposition of monoid congruences as a more faithful analog of primary decomposition.  
Second, we characterize which finite posets arise as the set of associated prime congruences of a congruence, and in doing so answer \cite[Problems~17.4 and~17.9]{kmmeso}.  

\subsection*{Acknowledgements}

The author is very grateful to Ezra Miller, Laura Matusevich, Thomas Kahle and Christine Berkesch for numerous discussions and conversations.  Much of this work was completed while the author was a graduate student at Duke University, funded in part by Ezra Miller's NSF Grant DMS-1001437.  Portions of this work also appeared in the author's doctoral thesis \cite{mesothesis}.

\section{Overview of mesoprimary decomposition of monoid congruences}
\label{s:mesodecomp}

In this section, we briefly review the necessary definitions and results from \cite{kmmeso} concerning mesoprimary decomposition of monoid congruences.  See \cite{kmmeso} for a more thorough treatment on mesoprimary decomposition, including the resulting (meso)primary decompositions of binomial ideals; see \cite{grilletSemigroups} for basic monoid definitions.  

\subsection*{Conventions}

Unless otherwise stated, $Q$ denotes a finitely generated (equivalently, noetherian) commutative monoid, and $\kk$ denotes an arbitrary field.  

\begin{defn}\label{d:cong}
A \emph{binomial} in $\kk[Q]$ is an element of the form $\ttt^a - \lambda\ttt^b$ where $a, b \in Q$ and $\lambda \in \kk$.  An ideal $I \subset \kk[Q]$ is \emph{binomial} (resp.~\emph{monomial}) if it can be generated by binomials (resp.~monomials).  An equivalence relation $\til$ on $Q$ is a \emph{congruence} if $a \sim b$ implies $a + c \sim b + c$ for all $a, b, c \in Q$.  The congruence~$\til_I$ on $Q$ \emph{induced} by a binomial ideal $I \subset \kk[Q]$  sets $a \sim_I b$ whenever $\ttt^a - \lambda\ttt^b \in I$ for some nonzero $\lambda \in \kk$.  
\end{defn}


\begin{notation}
For a congruence~$\til$ on~$Q$ and a prime $P \subset Q$, we write $Q_P$ for the localization along~$P$ and $\oQ_P = Q_P/\til$ for the quotient of $Q_P$ modulo $\til$.  We denote by~$\ol q$ the image of~$q \in Q$ in $\oQ = Q/\til$.  The nil of $Q$, if it exists, is denoted $\infty \in Q$.  
\end{notation}

\begin{defn}[{\cite[Definitions~2.12, 3.4, 4.7, 4.10, 7.1, 7.2, 7.7,
and~7.12]{kmmeso}}]\label{d:kmcong}
Fix a congruence~$\til$ on~$Q$ and a prime $P \subset Q$.  
\begin{enumerate}[(a)]
\item%
An element $q \in Q$ is an \emph{aide} for an element $w \in Q$ and a generator $p \in P$ if (i)~$\ol w \ne \ol q$, (ii)~$\ol w + \ol p = \ol q + \ol p$, and (iii)~$\ol q$ is maximal in the set $\{\ol q,\ol w\}$.  If $q$ is an aide for $w$ for each generator of $P$, then $q$ is a \emph{key aide}.  

\item 
An element $w \in Q$ is a \emph{witness} for~$P$ if it has an aide for each $p \in P$, and a \emph{key witness} for~$P$ if it has a key aide.  A key witness $w$ is a \emph{cogenerator} of $\til$ if $w + p$ is nil modulo $\til$ for all $p \in P$.

\item%
The congruence~$\til$ is \emph{$P$-primary} if every $p \in P$ is nilpotent in~$\oQ$ and every $f \in Q \minus P$ is cancellative in~$\oQ$.  A $P$-primary congruence~$\til$ is \emph{mesoprimary} if every element of the quotient~$\oQ$ is partly cancellative (that is, ).  The congruence~$\til$ is \emph{coprincipal} if it is mesoprimary and every cogenerator for $\til$ generates the~same~ideal in~$\oQ$.

\item%
The \emph{coprincipal component} $\til_w^P$ of~$\til$ cogenerated by a
witness $w \in Q$ for~$P$ is the coprincipal congruence that relates $a \sim_w^P b$ if one of the following is satisfied:
\begin{enumerate}[(i)]
\item%
both $\ol a$ and $\ol b$ generate an ideal not containing $\ol q$ in
$\oQ_P$; or
\item%
$\ol a$ and $\ol b$ differ by a unit in $\oQ_P$ and $\ol a + \ol c =
\ol b + \ol c = \ol q$ for some $\ol c \in \oQ_P$.
\end{enumerate}
\end{enumerate}
A (key) witness for~$P$ may be called a (key) $\til$-witness for~$P$
to specify~$\til$.  Congruences may be called $P$-mesoprimary or
$P$-coprincipal to specify~$P$.
\end{defn}

\begin{thm}[{\cite[Theorem~8.4]{kmmeso}}]\label{t:kmcong}
Each congruence~$\til$ on $Q$ is the common refine\-ment of the
coprincipal components cogenerated by its key witnesses.
\end{thm}

The proof of Theorem~\ref{t:kmcong} at the source \cite[Theorem~8.4]{kmmeso} implies the following.

\begin{cor}\label{c:kmcong}
Given a congruence~$\til$ on $Q$ and elements $a, b \in Q$ with $a
\nsim b$, there exists a monoid prime $P \subset Q$ and an element
$u \in Q$ such that (after possibly swapping $a$ and $b$) the element
$a + u$ is a key $\til$-witness for~$P$ with key aide $b + u$.
\end{cor}

Lastly, we recall the definition of prime congruences from \cite{kmmeso}, which play the role of ``associated objects'' in this setting.  

\begin{defn}[{\cite[Definitions~5.1 and~5.2]{kmmeso}}]\label{d:primecong}
Fix a congruence~$\til$ on a monoid~$Q$, a prime ideal $P \subset Q$,
and an element $q \in Q$ that is not nil modulo~$\til$.
\begin{enumerate}[(a)]
\item%
Let $G_P \subset Q_P$ denote the unit group of the localization~$Q_P$,
and write $K_q^P \subset G_P$ for the stabilizer of~$\ol q \in \oQ_P$
under the action of~$G_P$.

\item%
Let $\app$ denote the congruence on~$Q_P$ that sets $a \approx b$ when either (i)~$a$ and $b$ both lie in $P_P$ or (ii) $a$ and $b$ both lie in $G_P$ and $a - b \in K_q^P$.  The \emph{$P$-prime congruence} of~$\til$ at $q$ is $\ker(Q \to
Q_P/\app)$.

\item%
The $P$-prime congruence at $q$ is \emph{associated to $\til$} if $q$
is a key witness for~$P$.
\end{enumerate}
\end{defn}

\begin{remark}\label{r:mesoprimaryequiv}
By \cite[Corollary~6.7]{kmmeso}, a congruence is $P$-mesoprimary if and only if it is $P$-primary and the $P$-prime congruences at every non-nil element all coincide.  Generally speaking, each witness $w$ for $P$ detects an element whose $P$-prime congruence differs from those in the direction(s) of $P$, and the coprincipal component at $w$ distinguishes the $P$-prime congruence at $w$ from those above it in the decomposition in Theorem~\ref{t:kmcong}.  We direct the unfamiliar reader to \cite[Example~1.3]{kmmeso} and the accompanying graphics, which are a particularly enlightening illustration of mesoprimary decomposition at the level of congruences.  
\end{remark}

\section{True witnesses of monoid congruences}
\label{s:truewitnesses}

Key witnesses (Definition~\ref{d:kmcong}) form a restricted class of witnesses sufficient for decomposing any monoid congruence, but the coprincipal components they cogenerate may still be redundant (Example~\ref{e:keynottrue}).  
In this section, we restrict further to the class of true witnesses (Definition~\ref{d:truewitness}), which are still sufficient for decomposing any congruence (Theorem~\ref{t:truedecomp}).  

\begin{example}\label{e:keynottrue}
Let $I = \<x^3 - xy^2, x^3(z - 1), x^2y - y^3, y^3(w - 1), x^4, y^4\> \subset \kk[x,y,z,w]$.  Its congruence $\til_I$ on $Q = \NN^4$ is depicted in Figure~\ref{fig:keynottrue}, projected onto the $xy$-plane.  The congruence $\til_I$ is $P$-primary for $\mathfrak m_P = \<x,y\>$ and has five Green's classes of key witnesses, namely those containing the monomials $x^2$, $y^2$, $x^3$, $y^3$, and $x^3y$, respectively.  Indeed, $x^2$ and $y^2$ are each key aides for the other, $wx^3$ is a key aide for $x^3$, $zy^3$ is a key aide for $y^3$, and $x^3y$ has nil as a key aide.  Of these, $x^2$ and $y^2$ yield redundant components in the coprincipal decomposition for $\til_I$ in Theorem~\ref{t:kmcong}, and the remaining three comprise a mesoprimary decomposition for $\til_I$ with no redundant components.  
\end{example}

\begin{figure}[tbp]
\begin{center}
\includegraphics[width=1.5in]{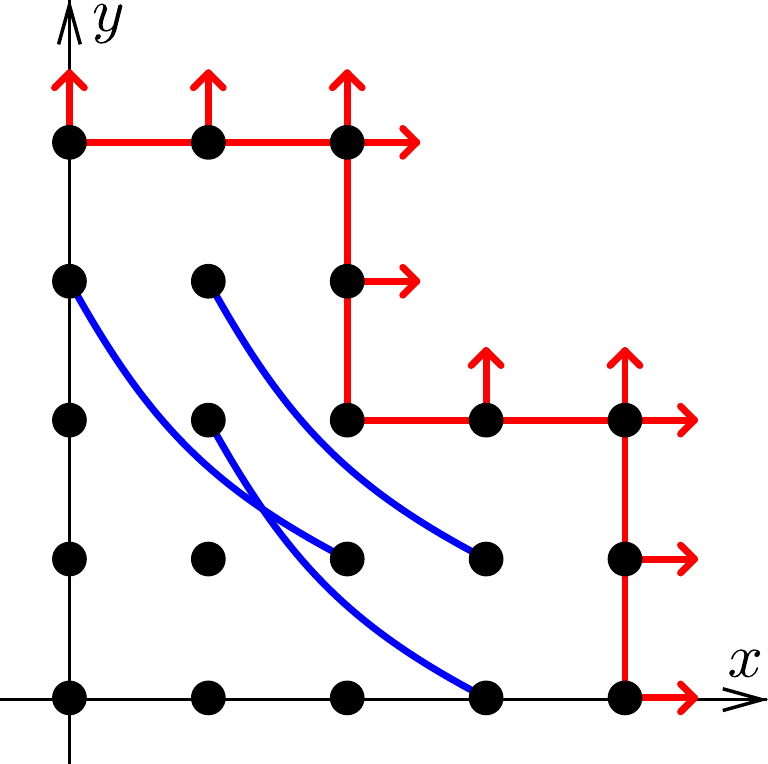}
\end{center}
\caption[A congruence with redundant key witnesses]{A congruence $\til$ on $\NN^4$ with key witnesses whose coprincipal components are redundant, projected onto the $xy$-plane.}
\label{fig:keynottrue}
\end{figure}

\begin{defn}\label{d:truewitness}
Fix a congruence $\til$ on $Q$, a prime $P \subset Q$, and an element $w \in Q$.  
\begin{enumerate}[(a)]
\item 
A \emph{$P$-cover congruence} of $w$ is the $P$-prime congruence at a non-nil element $w + p$ for some generator $p$ of $P$.  

\item 
The \emph{discrete testimony} of $w$ at $P$ is the set $T_P(w)$ of $P$-cover congruences of $w$.  The discrete testimony of $w$ is \emph{suspicious} if the common refinement of the $P$-cover congruences in the testimony coincides with the $P$-prime congruence at $w$.  

\item 
We say $w$ is a \emph{true witness} if either (i) $w$ is maximal among $\til$-witnesses for $P$, or (ii) the discrete testimony of $w$ is not suspicious.  

\item 
A $P$-prime congruence $\app$ is \emph{truly associated} to $\til$ if it is the $P$-prime congruence at a true $\til$-witness for $P$.  

\end{enumerate}
\end{defn}

\begin{example}\label{e:truewitnessconditions}
Conditions~(i) and~(ii) in Definition~\ref{d:truewitness}(c) are both necessary.  Indeed, 
$$I_1 = \<x^2 - xy, xy - y^2\> \subset \kk[x,y]$$
induces a congruence with two witnesses for the maximal prime $P$, both of which are maximal among witnesses for $P$ but neither of which has suspicious testimony since $\oQ$ has no nil element.  
Additionally, the congruence induced by
$$I_2 = \<z^4 - 1, x(z - 1), y(z^2 - 1), x^2, xy, y^2\> \subset \kk[x,y,z]$$
has three witnesses for the maximal prime $P$, one of which (the origin) has suspiciuos testimony but is not maximal among witnesses for $P$.  The congruences induced by $I_1$ and $I_2$ are depicted in Figure~\ref{fig:truewitnessconditions}.  
\end{example}

\begin{remark}\label{r:truevscharacter}
Character witnesses \cite[Definition~16.3]{kmmeso} are the binomial ideal analogues of true witnesses, except that their testimony is computed by intersecting ideals instead of refining congruences.  In general, however, character witnesses need not be true, and true witnesses need not be character.  Additionally, Corollary~\ref{c:keytrue} states that true witnesses are key, a fact that fails for character witnesses; see \cite[Examples~16.5--16.7]{kmmeso} for demonstration of this behavior.  
\end{remark}

\begin{figure}[tbp]
\begin{center}
\includegraphics[width=1.5in]{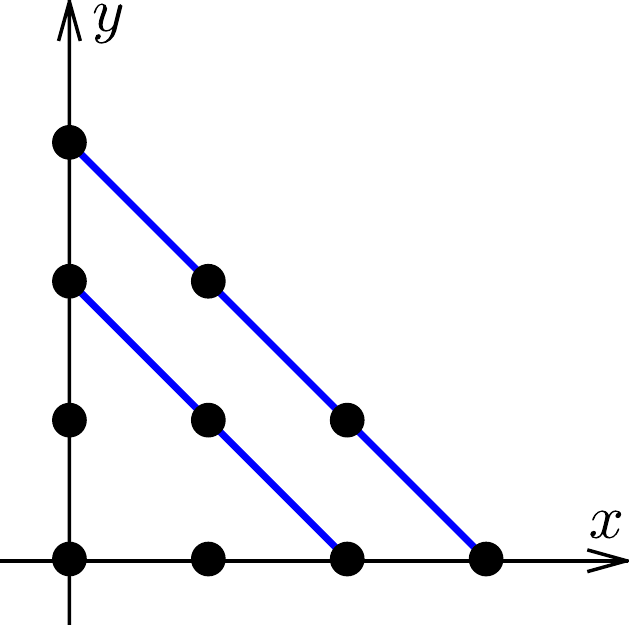}
\hspace{0.5in}
\includegraphics[width=1.5in]{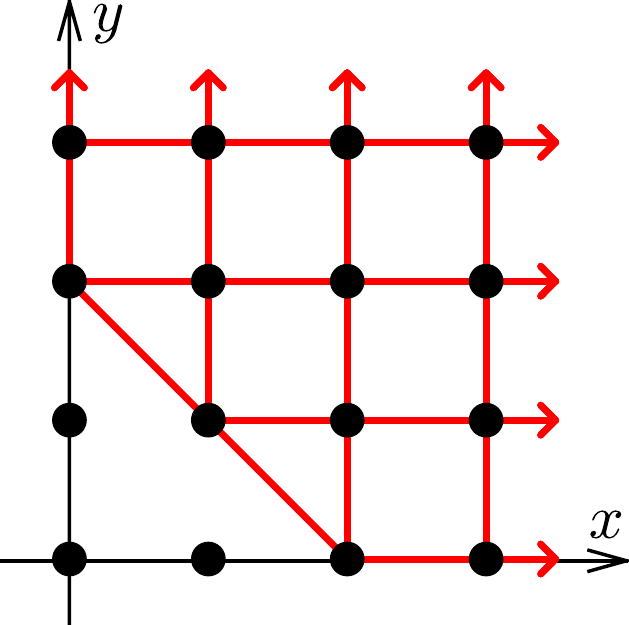}
\end{center}
\caption[Congruences with true witnesses]{Congruences for the ideals $I_1$ (left) and $I_2$ (right) in Example~\ref{e:truewitnessconditions}.}
\label{fig:truewitnessconditions}
\end{figure}

Proposition~\ref{p:trueaide} and Corollary~\ref{c:trueaide} each give an equivalent condition for identifying true witnesses that will be useful in proving Theorem~\ref{t:trulyassociatedcong}.  

\begin{prop}\label{p:trueaide}
Fix a congruence $\til$ on $Q$ and a witness $w$ for $P$.  The discrete testimony of $w$ is not suspicious if and only if $w$ has a key aide $w'$ that is either nil or generates the same ideal as $w$ in $Q_P$.  
\end{prop}

\begin{proof}
If $w$ has $\infty$ as a key aide, then its discrete testimony is empty.  If $w$ has a key aide $w'$ in its Green's class in $Q_P$, then each prime congruence in its discrete testimony identifies $w$ and $w'$, and thus so does their common refinement.  Either way, the discrete testimony of $w$ is not suspicious.  

Now suppose the discrete testimony of $w$ is not suspicious and that $\infty \in Q$ is not a key aide.  The set $T_P(w)$ is thus nonempty, and the common refinement of the prime congruences in $T_P(w)$ relates some $u$ and $v$ outside of $P$ that are not related under the prime congruence $\app$ at $w$.  This means any element $w'$ with $w + u = w' + v$ must satisfy $w + p = w' + p$ for each $p \in P$, making $w'$ a key aide for $w$.  
\end{proof}

\begin{cor}\label{c:trueaide}
The element $w$ in Proposition~\ref{p:trueaide} is a true witness if and only if $w$ either (i) is maximal among $P$-witnesses for $\til$, or (ii) has a key aide that generates the same ideal as $w$ in $Q_P$.  \qed
\end{cor}

\begin{cor}\label{c:keytrue}
Every true $\til$-witness is a key $\til$-witness, and every truly associated prime congruence of $\til$ is associated to $\til$.  \qed
\end{cor}

We are now ready for the main result of this section.  Theorem~\ref{t:truedecomp} shows that when constructing an induced coprincipal decomposition for a given congruence, it suffices to consider true witnesses.  In particular, any component in the decomposition given in Theorem~\ref{t:kmcong} cogenerated by a non-true witness is redundant and can be omitted.  

\begin{thm}\label{t:truedecomp}
Fix a congruence $\til$.  Every congruence on $Q$ is the common refinement of the coprincipal congruences cogenerated by its true witnesses.  
\end{thm}

\begin{proof}
Fix a congruence $\til$ on $Q$ and a key witness $w$ for $P$ that is not true.  In order to prove the congruence $\til_w^P$ is redundant in the decomposition $\til = \bigcap_i \til_i$ of Theorem~\ref{t:kmcong}, it suffices to produce, for $q, q' \in Q$ not identified under $\til_w^P$, a component~$\til_j \ne \til_w^P$ not identifying $q$ and $q'$.  Since primary decomposition of monoid congruences commutes with localization \cite[Theorem~3.12]{kmmeso}, it suffices to assume that $Q = Q_P$.  

First, suppose $q$ and $q'$ lie in distinct Green's classes in $Q$.  Since $w$ is not true, it is not maximal, so some maximal witness $v$ for $P$ lies above $w$.  The nil class of~$\til_v^P$ is properly contained in the nil class of $\til_w^P$, so $q$ and $q'$ are not both nil under $\til_v^P$.  Furthermore, outside of its nil class, $\til_v^P$ does not relate any elements that lie in separate Green's classes.  In particular, $\til_v^P$ does not relate $q$ and $q'$.  

Next, suppose $q$ and $q'$ lie in the same Green's class in $Q_P$.  Since $q$ and $q'$ are not both nil modulo $\til_w^P$, there exists $u \in Q$ such that $q + u$ and $q' + u$ are in the same Green's class as $w$.  Furthermore, any component that does not relate $q + u$ and $q' + u$ will not relate $q$ and $q'$, so upon replacing $q$ with $q + u$ and $q'$ with $q' + u$, it suffices to assume $u = 0$ and $q' = w$.  Since $w$ is not a true witness, $q$ is not a key aide for $w$, so $w + p \nsim q + p$ for some generator $p \in P$.  This means some component $\til_j$ does not relate $w + p$ and $q + p$, and thus does not relate $w$ and $q$, as desired.  
\end{proof}

\section{Irredundant mesoprimary decompositions of congruences}%
\label{s:minminimalcong}

In this section, we prove that each truly associated prime congruence of a given congruence $\til$ appears as the associated prime congruence of some mesoprimary component in every mesoprimary decomposition for $\til$ (Theorem~\ref{t:trulyassociatedcong}).  As a consequence, we prove that any congruence with no embedded associated monoid primes possesses both a unique minimal mesoprimary decomposition and a unique irredundant mesoprimary decomposition (Corollary~\ref{c:irredundantcong}).  Making statements about ``all'' mesoprimary decompositions necessitates some mild restrictions; see Remark~\ref{r:induced} and \cite[Example~8.2]{kmmeso}.  

\begin{defn}[{\cite[Definition~8.1]{kmmeso}}]\label{d:kmmesodecomp}
An expression $\til = \bigcap_i \app_i$ of a congruence~$\til$ as a common refinement of mesoprimary congruences is a \emph{mesoprimary decomposition} if, for each~$\app_i$ with associated prime $P_i$, the $P_i$-prime congruences of~$\til$ and~$\app_i$ at each cogenerator for~$\app_i$ coincide.  This decomposition is \emph{key} if every cogenerator for every~$\app_i$ is a key witness for~$\til$.
\end{defn}

\begin{remark}
Theorems~\ref{t:kmcong} and~\ref{t:truedecomp} both yield key mesoprimary decompositions.  
\end{remark}

\begin{defn}\label{d:congmesodecomp}
A mesoprimary decomposition $\til = \bigcap_i \til_i$ is 
\begin{enumerate}[(a)]

\item 
\emph{induced} if each $\til_i$ is a common refinement of coprincipal components;

\item 
\emph{minimal} if $\til_i$ and $\til_j$ have distinct associated prime congruences for $i \ne j$; or

\item 
\emph{irredundant} if no $\til_i$ can be omitted.  

\end{enumerate}
\end{defn}

\begin{remark}\label{r:induced}
The coprincipal component $\til_w^P$ of a congruence $\til$ at a witness $w$ for $P$ is determined by the congruence $\til$.  More precisely, it is the finest coprincipal congruence with cogenerator $w$ that can appear in a mesoprimary decomposition for~$\til$.  As~such, for the purpose of minimality, we restrict our attention to induced mesoprimary decompositions.  Indeed, if the induced condition is relaxed, coprincipal components whose cogenerator is a non-key $\til$-witness need not be redundant; see Example~\ref{e:congredundancy}.  
\end{remark}

\begin{example}\label{e:congredundancy}
The ideal $I = \<x^3y - x^2y^2, x^2y^2 - xy^3, x^5, y^5\>$ is the intersection of 
$I_1 = \<x^3y - x^2y^2, x^2y^2 - xy^3, x^4, y^4\>$ and $I_2 = \<x^2y - xy^2, x^5, y^5\>$.  
Their congruences $\til$, $\til_1$, and $\til_2$, respectively, are depicted in Figure~\ref{fig:ex202abc}.  
Both $\til$ and $\til_2$ are coprincipal with cogenerator $(4,1)$, but $\til_2$ is not 
the coprincipal component cogenerated by $(4,1)$ since it also identifies $(2,1)$ and $(1,2)$.  
As such, this mesoprimary decomposition is not induced.  Additionally, 
$\til_1$ is cogenerated by a non-key non-character witness for $\til$, 
but neither component of this mesoprimary decomposition can be omitted.  
\end{example}

\begin{figure}[tbp]
$\begin{array}{ccccc}
\includegraphics[width=1.2in]{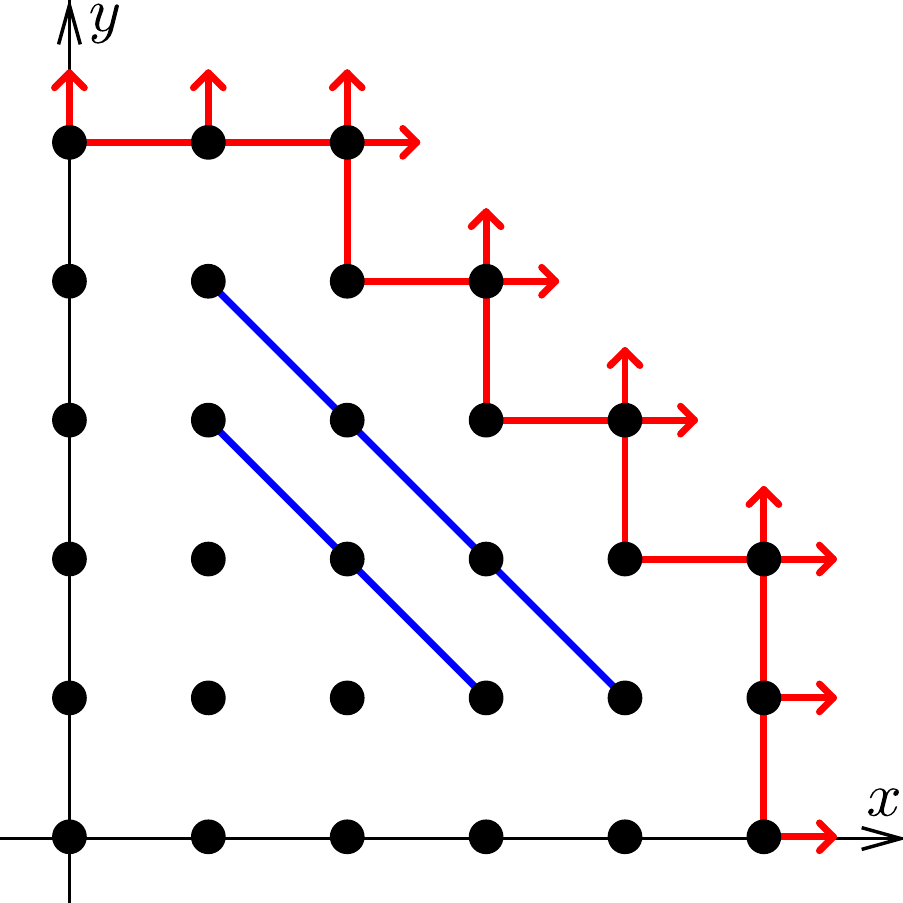}
&  & 
\includegraphics[width=1.2in]{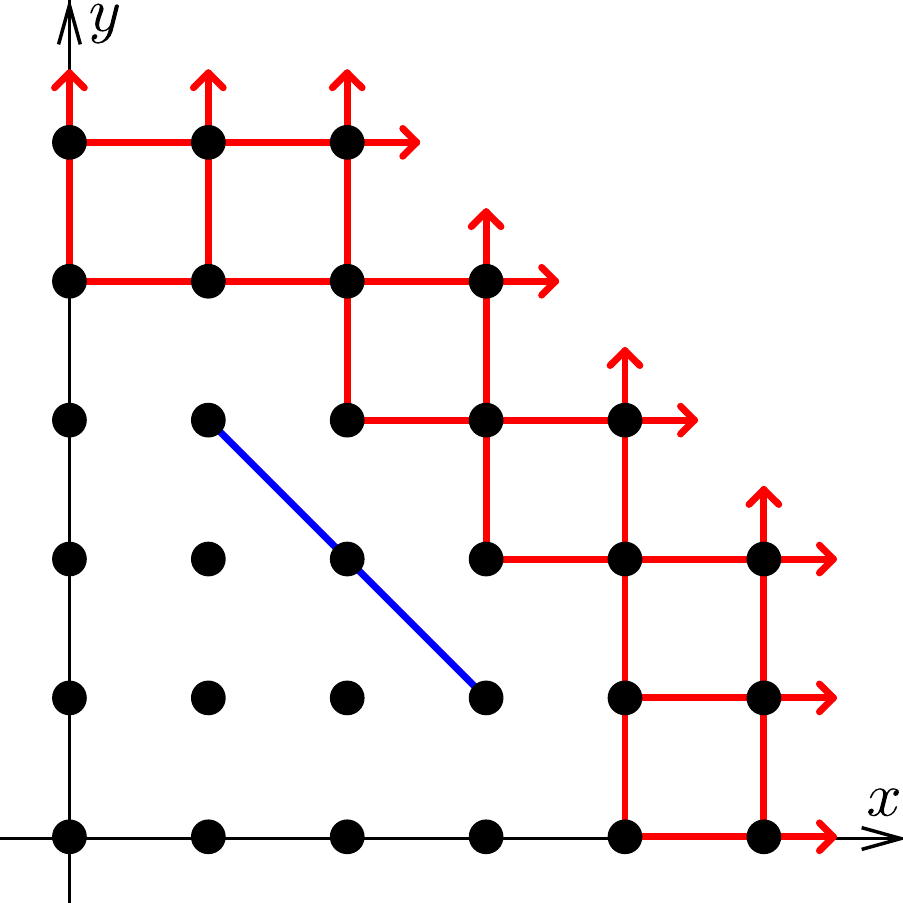}
&  & 
\includegraphics[width=1.2in]{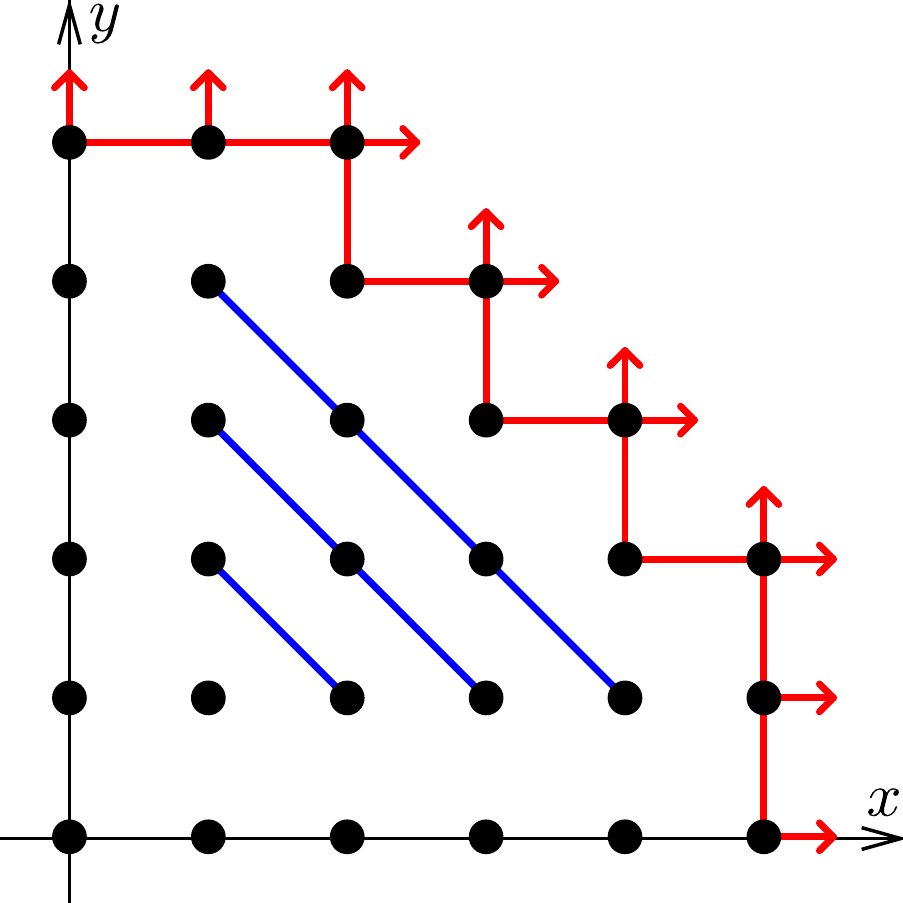}
\\
\til & = & \til_1 & \cap & \til_2
\end{array}$
\caption[A non-induced mesoprimary decomposition with a non-redundant component cogenerated by a witness that is not key]{The non-induced mesoprimary decomposition from Example~\ref{e:congredundancy}.}
\label{fig:ex202abc}
\end{figure}

An important observation is that any witness whose discrete testimony is not suspicious must appear as a cogenerator in every mesoprimary decomposition.  Notice the absence of ``induced'' here; we do indeed mean \emph{every} mesoprimary decomposition.  We~record this fact in Lemma~\ref{l:truecogen}, which serves as the foundation for Theorem~\ref{t:trulyassociatedcong}.  

\begin{lemma}\label{l:truecogen}
Fix a mesoprimary decomposition $\til = \bigcap_i \til_i$, and a $\til$-witness $w$ for $P$.  
If the discrete testimony of $w$ is not suspicious, then $w$ is a cogenerator for some $\til_i$.  
\end{lemma}

\begin{proof}
Let $\app$ denote the $P$-prime congruence at $w$, and let $\app_i$ denote the prime congruence associated to $\til_i$ for each $i$.  
By Proposition~\ref{p:trueaide}, either $w$ has $\infty$ as a key aide, or $w$ has a key aide $w'$ that is Green's equivalent to $w$ in the localization $Q_P$.  If $w$ has $\infty$ as a key aide, then it is a cogenerator for $\til$, so any mesoprimary component $\til_i$ under which $w$ is not nil also has $w$ as a cogenerator.  

Alternatively, suppose $w$ has a key aide $w'$ in the same Green's class as $w$ in $Q_P$.  Since $w \nsim w'$, some mesoprimary component $\til_i$ does not relate $w$ and $w'$.  Neither $w$ nor $w'$ is nil under $\til_i$, but for each generator $p$ of $P$, the prime congruence at $w + p$ relates $w$ and $w'$.  This means each $w + p$ must be nil under $\til_i$ because $\til_i$ is mesoprimary, so $w$ is a cogenerator for $\til_i$.  
\end{proof}

The symmetry in Example~\ref{e:symmetry}, which also appeared as \cite[Example~2.19]{kmmeso}, demonstrates that Lemma~\ref{l:truecogen} cannot be generalized to arbitrary true witnesses, as eliminating all redundancy sometimes requires making arbitrary choices.  That said, Lemma~\ref{l:maximalwitness} demonstrates that the phenomenon in Example~\ref{e:symmetry} is the only possible obstruction.  

\begin{example}\label{e:symmetry}
Let $I = \<x^2 - xy, xy - y^2\> \subset \kk[x,y]$.  The congruence $\til_I$ has two associated primes, namely $\emptyset$ and the maximal ideal $P$.  Theorem~\ref{t:kmcong} produces the coprincipal decomposition 
$$I = \<x^2 - xy, xy - y^2\> = \<x^2,y\> \cap \<x,y^2\> \cap \<x - y\>.$$
The first two components are $P$-primary, and the third is $\emptyset$-primary.  Either, but not both, of the first two components can be omitted without affecting the intersection, even though each is cogenerated by a true witness for $\til_I$.  
\end{example}

\begin{lemma}\label{l:maximalwitness}
Fix a congruence $\til$, a key $\til$-witness $w$ for $P$, and a key aide $w'$ for~$w$.  If $w$ is a maximal witness for $P$, then every mesoprimary decomposition $\til = \bigcap_i \til_i$ has a component with either $w$ or $w'$ as a cogenerator.  
\end{lemma}

\begin{proof}
Suppose $w$ is maximal among $\til$-witnesses for $P$.  Since primary decomposition of congruences commutes with localization by \cite[Theorem~3.12]{kmmeso}, it suffices to replace $Q$ with $Q_P$, so that $P$ is maximal.  If $w'$ is nil, then $w$ is a cogenerator for $\til$, so it is a cogenerator for any $P$-primary component $\til_i$ under which it is not nil.  If, instead, $w'$ lies in the same Green's class as $w$ in $Q_P$, then we are done by Lemma~\ref{l:truecogen}.  Lastly, assume $w'$ is not nil and lies in a different Green's class in $Q_P$.  Since $w \nsim w'$, some component $\til_i$ separates $w$ and $w'$.  Localization $Q$ at any prime $P'$ properly contained in $P$ identifies $w$ and $w'$ since $w + p = w' + p$ for any $p \in P \setminus P'$.  This means any $P'$-primary component also identifies $w$ and $w'$, so $\til_i$ must be $P$-primary.  Since $w$ is maximal among witnesses for $P$, it is either a cogenerator for $\til_i$ or nil modulo $\til_i$; the latter implies that $w'$ is a cogenerator for $\til_i$.  In either case, the proof is complete.  
\end{proof}

\begin{thm}\label{t:trulyassociatedcong}
Fix a congruence $\til$, a true $\til$-witness $w$ for a prime $P$, and let $\app$ denote the $P$-prime congruence at $w$.   
\begin{enumerate}[(a)]
\item 
If either (i)~the discrete testimony of $w$ is not suspicious, or (ii) the $P$-prime congruence at some non-nil key aide $w'$ for $w$ equals~$\app$, then $\app$ appears as the associated prime congruence of some mesoprimary component
in each mesoprimary decomposition $\bigcap_i \til_i$ of $\til$.  

\item 
If $w$ satisfies neither (i) nor (ii), then the component in the coprincipal decomposition in Theorem~\ref{t:truedecomp} 
with cogenerator $w$ is redundant.  

\end{enumerate}
\end{thm}

\begin{proof}
If the discrete testimony of $w$ is not suspicious, then apply 
Lemma~\ref{l:truecogen}.  On the other hand, if $w$ has a key aide $w'$ 
whose prime congruence is also $\app$, then by Lemma~\ref{l:maximalwitness} 
one of $w$ and $w'$ must appear as a cogenerator of some component $\til_i$.  
This~proves part~(a).  

Next, fix $a, b \in Q$ with $a \nsim b$.  By Corollary~\ref{c:kmcong}, 
there is a prime $P \subset Q$ and $u \in Q$ such that (after possibly swapping $a$ and $b$) $a + u$ is a key witness with key aide $b + u$.  If $a + u$ has suspicious discrete testimony, then by Proposition~\ref{p:trueaide} it does not have nil as a key aide, so $b + u$ is also a key witness for $P$.  If, additionally, $a + u$ and $b + u$ have distinct $P$-prime congruences, then since $a + u$ and $b + u$ have identical discrete testimony, the discrete testimony of $b + u$ is not suspicious.  Since $a \nsim_{b+u}^P b$, this completes the proof.  
\end{proof}


\begin{cor}\label{c:completeanalog}
Fix a mesoprimary decomposition $\til = \bigcap_i \til_i$.  Each truly associated prime congruence in of $\til$ is associated to some component $\til_i$, and any component whose associated prime congruence is not truly associated to $\til$ is redundant.  \qed
\end{cor}

We conclude this section by characterizing the minimal and irredundant mesoprimary decompositions of congruences with no embedded associated monoid primes.  

\begin{thm}\label{t:mintruecogen}
Fix a mesoprimary decomposition $\til = \bigcap_i \til_i$.  If $P$ is a minimal associated prime of $\til$, then every true witness $w$ of $P$ 
is a cogenerator of some component.  
\end{thm}

\begin{proof}
Let $\app$ denote the $P$-prime congruence at $w$, and let $\app_i$ denote the prime congruence associated to $\til_i$ for each $i$.  If $P = \emptyset$, then since $P$ is associated to $\til$, some component $\til_i$ is $P$-primary, and in fact $\til_i = \app$.  Now assume $P$ is nonempty.  Once again, after localizing at $P$, assume $P$ is maximal.  Since $P$ is a minimal associated prime, $\til$ is $P$-primary by \cite[Corollary~4.21]{kmmeso}.  Since $w$ is true, either it is a maximal witness for $P$, in which case it has $\infty$ as a key aide, or its testimony is not suspicious.  In either case, we are done by Lemma~\ref{l:truecogen}.  
\end{proof}

\begin{cor}\label{c:irredundantcong}
Any congruence $\til$ on $Q$ with no embedded associated monoid primes has a unique irredundant induced coprincipal decomposition and a unique induced mesoprimary decomposition.  In particular, this holds when $\til$ is primary.  
\end{cor}

\begin{proof}
Theorem~\ref{t:truedecomp} produces the unique induced coprincipal decomposition, as omitting any component yields an expression that cannot decompose $\til$ by Theorem~\ref{t:mintruecogen}.  Furthermore, replacing any set of components with their common refinement whenever they share an associated prime congruence results in a minimal mesoprimary decomposition by \cite[Proposition~6.9]{kmmeso}.  
\end{proof}

\section{Posets of associated mesoprimes}\label{s:posets}

In the final section of this paper, we answer a question posed by Kahle and Miller.  It is known that any poset occurs as the set of associated primes of a monomial ideal; as such, the question is posed only for primary congruences, so that the nilpotent directions of the associated prime congruence (i.e.\ the ``monomial part'' of an ideal inducing the congruence) all coincide.  

\begin{prob}[{\cite[Problem~17.4]{kmmeso}}]\label{prob:congposets}
Characterize the posets of associated prime congruences of primary congruences.  
\end{prob}

Theorem~\ref{t:easyposetanswer} provides a full, albeit unsatisfying, answer to Problem~\ref{prob:congposets} as stated.  The issue is that in the constructed congruence, most of the witnesses are incomparable under the divisibility poset of $Q$.  In view of this, we introduce the prime congruence poset (Definition~\ref{d:primecongruenceposet}), which only renders associated prime congruences comparable if they occur at comparable elements under divisibility in $\oQ_P$.  Surprisingly, the prime congruence poset has no further restrictions than the poset of truly associated prime congruences (Theorem~\ref{t:posetanswer}).  

Note that the content of this section also answers \cite[Problem~17.9]{kmmeso}; see Remark~\ref{r:mesoprimeposets}.  

\begin{defn}\label{d:assposet}
The \emph{poset of truly associated prime congruneces} of a congruence $\til$ is 
$$\MesoAss(\til) = \{\app \text{ truly associated to } \til\},$$
partially ordered by refinement.  
\end{defn}


\begin{lemma}\label{l:uniquemin}
For any primary congruence $\til$, $\MesoAss(\til)$ has a unique minimum.  
\end{lemma}

\begin{proof}
The prime congruence $\app$ at the origin refines the prime congruence at every non-nil element, and any nilpotent element that is maximal among those with prime congruence $\app$ is a true witness by Theorem~\ref{p:trueaide}.  As such, $\app \in \MesoAss(\til)$.  
\end{proof}




\begin{prop}\label{p:mesoprimeposet}
Fix a finite poset $\Omega = \{p_0, \ldots, p_d\} \subset 2^{[n]}$.  Fix distinct primes $a_0, \ldots, a_n \in \ZZ$, and for $0 \le i \le d$, let 
$$b_i = \prod_{j \in p_i} a_j \quad \text{ and } \quad I_i = \<y^{b_i} - 1\> \subset \kk[y].$$
\begin{enumerate}[(a)]
\item 
The posets (i) $\{b_0, \ldots, b_d\}$, ordered by divisibility, and (ii) $\{I_0, \ldots, I_d\}$, ordered by reverse containment, each coincide with $\Omega$.  

\item 
No ideal $I_i$ equals the intersection of a collection of ideals in $\{I_0, \ldots, I_d\} \setminus \{I_i\}$.  

\end{enumerate}
\end{prop}

\begin{proof}
This follows from the fact that $(y^c - 1) \mid (y^{c'} - 1)$ if and only if $c \mid c'$.  
\end{proof}

\begin{thm}\label{t:easyposetanswer}
Fix a finite subset $\Omega = \{A_0, A_1, \ldots, A_d\} \subset 2^{[n]}$ with $A_0 = \emptyset$.  Let $I_0, \ldots, I_d \subset \kk[z]$ denote the ideals from Proposition~\ref{p:mesoprimeposet}.  If
$$I = I_0 + x_1I_1 + \cdots + x_dI_d + \<x_1, \ldots, x_d\>^2 \subset \kk[x_1, \ldots, x_d, y]$$
then the poset $\MesoAss(\til_I)$ is isomorphic to $\Omega$.  
\end{thm}

\begin{proof}
For each $i \in \{1, \ldots, d\}$, the monoid element corresponding to $x_i$ is a key witness for $\til_I$ with associated prime congruence induced by $I_i$, and the prime congruence at the origin is the congruence incuded by $I_0$.  As such, $\MesoAss(\til_I) = \{\til_0, \til_1, \ldots, \til_d\}$ is isomorphic to the poset $\Omega$.  
\end{proof}

\begin{defn}\label{d:primecongruenceposet}
Fix a primary congruence $\til$ on $Q$.  Given $q \in Q$, 
let $\app_q$ denote the $P$-prime congruence of $\til$ at $q$.  The \emph{prime congruence poset} $(\Omega(\til), \preceq)$ consists of
\begin{itemize}
\item 
the set $\Omega(\til)$ of pairs $(q,\app_q)$ for non-nil $q \in \oQ_P$ modulo the equivalence relation generated by relating $(a, \app_a)$ and $([a + b], \app_{a+b})$ whenever $\app_a = \app_{a+b}$, and

\item 
the partial ordering $\preceq$ under which $(a, \app_a) \preceq (b, \app_b)$ whenever $\<a\> \supset \<b\>$.  

\end{itemize}
\end{defn}

\begin{lemma}\label{l:posrel}
Fix a primary congruence $\til$  The poset relation of $\til$ is an equivalence relation which coarsens $\til$, and the order $\preceq$ on $\Omega(\til)$ is a partial order.  
\end{lemma}

\begin{proof}
The important observation is that when $\<a\> \supset \<b\>$ for non-nil $a, b \in \oQ_P$, the prime congruence $\app_a$ at $a$ coarsens the prime congruence $\app_b$ at $b$.  This implies 
\begin{enumerate}[(i)]
\item 
$(a, \app_a)$ and $(b, \app_b)$ are identified in $\Omega(\til)$ whenever $a$ and $b$ lie in the same Green's class in $\oQ_P$, and 

\item 
if $\<a\> \supset \<b\>$ and $\app_a = \app_b$, then the prime congruence $\app_c$ at any $c$ satisfying $\<a\> \supset \<c\> \supset \<b\>$ agrees with $\app_a$ and $\app_b$.  

\end{enumerate}
As such, any pairs $(a,\app_a)$ and $(b, \app_b)$ identified in $\Omega(\til)$ do indeed satisfy $\app_a = \app_b$.   

At this point, checking that $\preceq$ is a partial order is straightforward.  Clearly $\preceq$ is reflexive, and transitivity of $\preceq$ follows from transitivity of Green's preorder on $Q$ and the transitivity of the equivalence relation defining $\Omega(\til)$.  Lastly, if $(a,\app_a) \preceq (b,\app_b)$ and $(b, \app_b) \preceq (a, \app_a)$, then the obserations in the above paragraph imply $\app_a = \app_b$, meaning $(a, \app_a)$ and $(b, \app_b)$ are identified in $\Omega(\til)$.  This completes the proof.  
\end{proof}

\begin{thm}\label{t:posetanswer}
Fix a finite subset $\Omega = \{p_0, p_1, \ldots, p_d\} \subset 2^{[n]}$ with $p_0 = \emptyset$.  Let $I_0, \ldots, I_d$ denote the ideals from Proposition~\ref{p:mesoprimeposet}, and define 
$$M = \<x_1^2, \ldots, x_d^2\> + \<x_ix_j : p_i, p_j \text{ incomparable}\>$$
and $B = \<x_ix_j - x_ix_k : p_i \supset p_j \text{ and } p_i \supset p_k\>$.  The ideal 
$$I = B + x_1I_1 + \cdots + x_dI_d + M \subset \kk[x_1,\ldots, x_d, y]$$
has $\Omega(\til_I)$ isomorphic to $\Omega$.  
\end{thm}

\begin{proof}
The only monomials in the variables $x_1, \ldots, x_d$ that lie outside of $I$ are either degree 1 or have the form $x_ix_j$ for $p_i \supset p_j$ (in particular, $I$ contains every monomial of total degree~3).  The only prime congruences that occur are induced by $I_0, \ldots, I_d$; $I_0$ induces the prime congruence at the origin, and $I_i$ for $i \ge 1$ induces the prime congruence at the elements corresponding to the monomials $\{x_i, x_ix_j : p_i \supsetneq p_j\}$.  The binomials generating $B$ ensure this set has exactly two distinct elements modulo $I$, the larger of which corresponds to the unique true witness whose associated prime congruence is induced by $I_i$.  Divisibility among the nonzero monomials modulo $I$ ensures $\Omega(\til)$ is isomorphic to $\Omega$.  
\end{proof}


\begin{remark}\label{r:mesoprimeposets}
In general, the set of associated prime congruences (as well as the prime congruence poset) can differ if different classes of witnesses are used in place of true witnesses (e.g.\ they may have different cardonalities).  However, every witness for every congruence constructed in Theorems~\ref{t:easyposetanswer} and~\ref{t:posetanswer} is true.  This means if one relaxed the problem to allow prime congruences at any more general class of witnesses, the resulting poset would be the same.  Consequently, the content of this section also answers \cite[Problem~17.9]{kmmeso}, the analog of Problem~\ref{prob:congposets} for binoimal ideals.  Indeed, upon referencing \cite[Definitions~10.4 and~12.1]{kmmeso}, one can easy check that each ideal $I$ defined in Theorems~\ref{t:easyposetanswer} or~\ref{t:posetanswer} decomposes as an intersection of mesoprimary ideals whose poset of associated mesoprimes is also isomorphic to the given poset $\Omega$.  
\end{remark}


\end{document}